\documentclass{amsart}
\usepackage{amsmath, amscd, amssymb, latexsym, hhline}
\usepackage[all]{xy}
\usepackage{enumitem}

\def\a{{\alpha}}
\def\b{{\beta}}
\def\e{{\epsilon}}
\def\w{{\omega}}
\def\k{{\Bbbk}}
\def\l{{\lambda}}
\def\s{{\sigma}}
\def\Lam{{\Lambda}}

\def\MM{\mathcal{M}}
\def\OO{\mathfrak{O}}
\def\AA{\mathcal{A}}

\def\lbidu{{}^{\vee\vee}}
\def\ldu{{}^{\vee}}

\newcommand\Irr{\operatorname{Irr}}
\newcommand\Soc{\operatorname{Soc}}
\newcommand\head{\operatorname{Head}}
\newcommand\ord{\operatorname{ord}}

\newcommand\Tr{\operatorname{Tr}}

\newcommand\id{\operatorname{id}}
\DeclareMathOperator\ev{{\operatorname{ev}}}
\DeclareMathOperator\db{\operatorname{db}}
\def\to{\rightarrow}
\def\Hom{{\mbox{\rm Hom}}}
\def\End{{\mbox{\rm End}}}
\newcommand\enumeri[1]{\begin{enumerate}[label=\rm(\roman*), leftmargin=*] #1 \end{enumerate}}
\newcommand\enumera[1]{\begin{enumerate}[label=\rm(\arabic*), leftmargin=*] #1 \end{enumerate}}
\newcommand\C[1]{#1\mbox{-\bf{mod}}_{\operatorname{\mathsf{fin}}}}
\newcommand\ld[1]{{}_{#1}}
\newcommand\lu[1]{{}^{#1}}
\newcommand\inv{^{-1}}
\newcommand\bidu{^{\vee\vee}}
\newcommand\quaddu{^{\vee\vee\vee\vee}}
\newcommand\du{^{\vee}}
\newcommand\ol[1]{\overline{#1}}
\renewcommand\o{\otimes}
\newtheorem{thm}{Theorem}[section]
\newtheorem{cor}[thm]{Corollary}
\newtheorem{prop}[thm]{Proposition}
\newtheorem{lem}[thm]{Lemma}

\theoremstyle{definition}

\theoremstyle{remark}
\newtheorem{remark}{Remark}[section]

\makeatletter
\def\namelabel#1#2{\@bsphack
  \protected@write\@auxout{}%
         {\string\newlabel{#1.nme}{{#2}{#2}}}%
  \@esphack}
\def\nmlabel#1#2{\label{#2}\namelabel{#2}{#1}}
\newcommand\nmref[1]{\ref{#1.nme}\ \ref{#1}}
\makeatother

\title{Hopf algebras of dimension $pq$, II}
\author{Siu-Hung Ng}
\address{Department of Mathematics, Iowa State University, Ames, IA 50011, USA.}
 \email{rng@iastate.edu}
 \thanks{The research was partially supported by NSA grant no. H98230-05-1-0020.}
\begin{document}
\maketitle
\begin{abstract}
  Let $H$ be a Hopf algebra of dimension $pq$ over an algebraically
  closed field of characteristic zero, where $p, q$ are odd primes with $p< q \le 4p+11$.
  We prove that $H$ is semisimple and thus isomorphic to a group algebra, or the dual of a group algebra.
\end{abstract}
\section*{Introduction}
In recent years, there has been some progress on the classification
of finite-dimensional Hopf algebras of dimension $pq$ over an
algebraically closed field $\k$ of characteristic 0, where $p,q$ are
prime numbers.  The case for $p=q$ has been settled completely in
\cite{Ng02}, \cite{Ng02a} and \cite{Mas96}. They are group algebras
and Taft algebras  of dimension $p^2$ (cf. \cite{Taft71}). However,
the classification for the case $p \ne q$  remains open in general.

In the works \cite{EG99} and \cite{GW00}, the semisimple case of the
problem has been solved; namely, any semisimple Hopf algebra of
dimension $pq$ is isomorphic to  a group algebra or the dual of a
group algebra. It is natural to ask whether Hopf algebras of
dimension $pq$, where $p<q$ are prime numbers, are always
semisimple. The question has been answered affirmatively in some
specific low dimensions. Williams settled the dimensions 6 and 10 in
\cite{WI}, Andruskiewitsch and Natale did dimensions 15, 21, and 35
\cite{AN1}, and Beattie and D{\u{a}}sc{\u{a}}lescu did dimensions
14, 55, 65, and 77 \cite{BD04}.

    More recently, the author has proved that if $p, q$ are
twin primes, or  $p=2$, then Hopf algebras of dimension $pq$ are
semisimple (cf. \cite{Ng04} and \cite{Ng05}). Meanwhile, Etingof and
Gelaki proved the same result for $2 < p < q \le 2p+1$ 
by considering the indecomposable projective modules over these Hopf
algebras (cf. \cite{EG03}).

In this paper, we will study the indecomposable modules over these
Hopf algebras. We prove that if $p, q$ are odd primes and $p < q \le
4p+11$, then every Hopf algebra of dimension $pq$ over $\k$ is
semisimple (Theorem \ref{t:main}). The result covers all odd
dimensions listed above.

The organization of the paper is as follows: we begin with some
notations and preliminary results for modules over a
finite-dimensional Hopf algebra in Section 1. In Section 2, we
obtain a lower bound for the dimensions of certain indecomposable
modules over a non-semisimple Hopf algebra $H$ of dimension $pq$. In
Section 3, we further assume $H$ is not unimodular, and consider the action of the group
$G(H^*)$ of all group-like elements of $H^*$ on the set $\Irr(H)$ of
isomorphism classes of simple $H$-modules. We obtain some lower
bounds for the number of $G(H^*)$-orbits in $\Irr(H)$. We finally
prove our main result in Section 4.

Throughout this paper, the base field $\k$ is always assumed to be
algebraically closed of characteristic zero and the tensor product
$\otimes$ means $\otimes_\k$, unless otherwise stated. The notation
introduced in Section 1 will continue to be used in the remaining
sections. The readers are referred to \cite{Mont93bk} and
\cite{Sw69} for elementary properties of Hopf algebras.
\section{notations and preliminaries}
Let $H$ be a finite-dimensional Hopf algebra over  $\k$ with
antipode $S$, counit $\e$. There are natural actions
$\rightharpoonup$ and $\leftharpoonup$ of $H^*$ on $H$ given by
$$
f \rightharpoonup a = a_1 f(a_2), \quad \text{and}\quad a
\leftharpoonup f = f(a_1)a_2
$$
where $\Delta(a)=a_1 \o a_2$ is Sweedler's notation with the
summation suppressed. For $f \in H^*$, we define the $\k$-linear
endomorphism $L(f)$ and $R(f)$ on $H$ by
$$
L(f)(a)= f \rightharpoonup a, \quad\text{and}\quad R(f)(a)= a
\leftharpoonup f.
$$

A non-zero element $a$ of $H$ is said to be group-like if
$\Delta(a)=a \o a$, and we denote by $G(H)$ the set of all
group-like elements of $H$. For $\b \in G(H^*)$,  $\b$ is an algebra
epimorphism from $H$ onto $\k$. The associated maps $L(\b)$, $R(\b)$
are algebra automorphisms of $H$, and they commute with $S^2$.
Moreover, each $\b \in G(H^*)$ is a degree 1 irreducible character
of $H$. We will denote by $\k_\b$ the 1-dimensional $H$-module which
affords the character $\b$. In particular, $\k_\e$ is the trivial
$H$-module $\k$.

Let $\Lam$ be a non-zero left integral of $H$.  The distinguished
group-like element  $\a$ of $H^*$ is defined by $\Lam a  =
\a(a)\Lam$ for $a \in H$. Similarly, if $\l \in H^*$ is a non-zero
right integral of $H^*$, the distinguished group-like element $g$ of
$H$ is defined by $f*\l  = f(g) \l $ for all $f \in H^*$. In this
convention, the celebrated Radford formula \cite{Radf76} is given by
\begin{equation}\label{eq:Rad_ant}
  S^4(h) = g(\a\rightharpoonup h \leftharpoonup\a\inv)g\inv \quad
  \text{for }h \in H\,.
\end{equation}
The non-zero right integral  $\l$ defines a non-degenerate
associative bilinear form on $H$, and so $H$ is a Frobenius algebra.
By \cite{Radf94},
\begin{equation}\label{eq:nakayama}
  \l(ab)=\l(\theta(b)a) \quad \text{for } a, b \in H\,,
\end{equation}
where $\theta(b)= S^2(b \leftharpoonup \a)$. Therefore, $\theta\inv$
is the associated \emph{Nakayama automorphism}.

Recall that the \emph{left dual} $V\du$ of an $H$-module $V$ is the
left $H$-module with the underlying space $V^*=\Hom_\k(V,\k)$ and
the $H$-action given by
$$
(hf)(x)= f(S(h)x)\quad \mbox{for}\quad x \in V, \quad f \in V^*\,.
$$
Similarly, the \emph{right dual} $\lu{\vee}V$ of $V$ is the left
$H$-module $V^*$ with the same underlying space $V^*$ and the
$H$-action given by
$$
(hf)(x)= f(S\inv(h)x)\quad \mbox{for}\quad x \in V, \quad f \in
V^*\,.
$$

Given an algebra automorphism $\sigma$ on $H$, one can \emph{twist}
the action of an $H$-module $V$ by $\sigma$ to obtain another
$H$-module $\ld{\sigma}V$. More precisely,  $\ld{\sigma}V$ is an
$H$-module on $V$ with the action given by
$$
h \cdot_\sigma v = \sigma(h)v\quad \text{for } h \in H,\, v \in V.
$$
$\ld\sigma (-)$ defines a $\k$-linear equivalence on the category
$\C{H}$ of finite-dimensional left $H$-module. In particular, if
$\sigma$ is an inner automorphism of $H$, then $\ld \sigma (-)$ is
$\k$-linearly equivalent to the identity functor.

Using the above notation, one can easily see that $\rho: \k_{\b} \o
V\bidu \to \ld{S^2 \circ R(\b)}V$ defined by
\begin{equation}\label{eq:rho}
\rho(1 \o \hat{v}) = v
\end{equation}
is a natural isomorphism of $H$-modules for $V \in \C{H}$ and $\b
\in G(H^*)$, where $\hat{v} \in V^{**}$ denotes the natural image of
the element $v \in V$. In particular, we have the $H$-module
isomorphisms
$$
V\bidu \cong  \ld{S^2}V,  \quad\text{and} \quad \k_\a \o V\bidu
\cong \ld{\theta}V\,.
$$
Similarly, one also have the $H$-module isomorphisms
$$
 \lbidu V \cong \ld{S^{-2}}V \quad\text{and} \quad \k_{\a\inv} \o
\lu{\vee\vee}V \cong \ld{\theta\inv}V\,.
$$
By Radford's antipode formula \eqref{eq:Rad_ant}, we also have
\begin{equation}\label{eq:swap}
 \k_{\a}\o V\quaddu \o \k_{\a\inv} \cong \k_{\a}\o \ld{S^{4}}V \o
 \k_{\a\inv}
 \cong V\,.
\end{equation}

It follows from a property of Frobenius algebras that an $H$-module
is projective if, and only if, it is injective. Each indecomposable
projective $H$-module $P$ is isomorphic to $He$ for some primitive
idempotent $e \in H$. Moreover, the socle $\Soc(P)$ and the head
$\head(P)=P/JP$  of  $P$ are simple (cf. \cite[IX]{CR62book}). Let us denote the projective
cover and the injective envelope of an $H$-module $V$ by $P(V)$ and
$E(V)$ respectively. If $V$ is a simple $H$-module, then $\ol V=
\Soc(P(V))$ is also simple, and so $P(V) \cong E(\ol V)$. The
assignment of simple $H$-module  $V \mapsto \Soc(P(V))$ defines a
permutation on a complete set of non-isomorphic simple $H$-modules,
and its inverse $\pi$ is called the \emph{Nakayama permutation} (cf.
\cite[\S 16A]{Lam99book}), i.e. $P(\pi(V))\cong E(V)$.  By \cite[\S
16C]{Lam99book}, if $P(V) \cong He$ for some primitive idempotent $e
\in H$, we have
\begin{equation}\label{eq:nak_per}
P(\pi(V)) \cong H\theta\inv(e)
\end{equation}
and so $V \cong \Soc(H \theta\inv(e))$. This allows us to rewrite
the Nakayama permutation as in the following lemma.
\begin{lem}\label{p:nakayama perm}
  Let $H$ be a finite-dimensional Hopf algebra over $\k$ with
  distinguished group-like element $\a \in H^*$, and
  $V$  a simple $H$-module. Then $\pi(V) \cong \ld \theta V$
  and hence
  \begin{equation}\label{eq:soc p}
    \k_{\a\inv} \o \lu{\vee\vee}V \cong \Soc (P(V))\quad\text{and} \quad
  V \cong \Soc (P(\k_\a \o V\bidu))\,.
  \end{equation}
  Moreover, $V$ is projective if, and only if, $\dim P(V)< 2\dim
  V$. In this case, $V \cong \k_\a \o V\bidu$ as $H$-modules.
\end{lem}
\begin{proof} Let $e$ be a primitive idempotent of $H$ such that
$\head (He) \cong V$. Note that $\ld \theta (He) \cong
  H\theta\inv(e)$ as $H$-modules under the map
  $he  \mapsto \theta\inv(he)$. By \eqref{eq:nak_per},
  $$
  \pi(V) \cong \head \ld\theta (He) \cong \ld \theta V\,.
  $$
  Since $\ld \theta V \cong \k_\a \o V\bidu$, the second isomorphism of \eqref{eq:soc p}
  follows  immediately from the definition of Nakayama
  permutation. Since $\k_{\a\inv} \o \lbidu(\k_\a \o V\bidu) \cong
  V$, we have $\pi\inv(V) \cong \k_{\a\inv} \o \lbidu V$ and hence
  $\k_{\a\inv} \o \lbidu V \cong \Soc(P(V))$.

  If $V$ is not projective, then
  $$
  \dim P(V) \ge \dim \Soc(P(V)) +\dim \head(P(V)) =2\dim V\,.
  $$
  Obviously, if $V$ is projective, then $P(V)=V$ and hence $\dim P(V) <2 \dim
  V$. In this case, $\k_\a \o V\bidu$ is also projective. Therefore,
  we have
  $$
  V \cong \Soc P(\k_\a \o V\bidu) \cong \k_\a \o V\bidu. \qedhere
  $$
\end{proof}

\begin{lem} \label{l:trace0} Let $V \in \C{H}$ such that $V \cong
\k_\b \o V\bidu$ for some non-trivial  $\b \in G(H^*)$. Then
$\Tr(\tau)=0$ for $\tau \in \Hom_H(V, \ld{\sigma}V)$, where $\sigma
=S^2\circ R(\b)$.
\end{lem}
\begin{proof}
Note that the evaluation map $\ev : V\du \o V \to \k$ and the
diagonal basis map $\db: k \to V \o V\du$ defined by $\db(1)=\sum_i
x_i \o x^i$ are $H$-module homomorphisms, where $\{x_i\}$ is a basis
for $V$ and $\{x^i\}$ is the dual basis for $V^*$. Consider the
composition
  \begin{equation}\label{eq:comp}
 \k \xrightarrow{\db}  V \o V\du \xrightarrow{\tau \o \id} \ld{\sigma}V \o
 V\du
 \xrightarrow{\rho\inv \o \id} \k_\b \o V\bidu \o V\du
 \xrightarrow{\id \o \ev} \k_\b
  \end{equation}
  of $H$-module maps, where $\rho$ is defined in \eqref{eq:rho}.
   The composition is a scalar given by
  $$
  \sum_i x^i (\tau(x_i)) =\Tr(\tau).
  $$
  Since $\k$ and $\k_\b$ are not isomorphic as
  $H$-modules, the composition is a zero map and hence
  $$
  \Tr(\tau)=0\,.\qedhere
  $$
  \end{proof}
  The composition displayed in \eqref{eq:comp} may not be zero if $\b$ is trivial.
  However, if $H$ is not semisimple and $V$ is projective, then such composition must be
  zero. For otherwise, $\db: \k \to V\o V\du$ is a split embedding and
  hence $\k$ is a direct summand of the projective $H$-module $V \o
  V\du$. This implies $\k$ is projective. However, $\k$ is not projective if $H$ is
  not semisimple (cf. \cite[Lemma 2.2]{EG03}). The reason for this observation has been presented in the
  proofs of \cite[Lemma 2.11]{EG03}, \cite[Theorem 2.16]{EO04} and
  \cite[Theorem 2.3 (b)]{Lo97}. We state this conclusion as
\begin{lem} \label{l:2ndtrace0}
  Let $V$ be a projective module over a non-semisimple finite-dimensional Hopf algebra $H$
  such that $V \cong V\bidu$. Then for all homomorphism $\phi: V \to
  \ld{S^2} V$ of $H$-modules, $\Tr(\phi)=0$. \qed
\end{lem}
We close this section with a generalization of \cite[Proposition
2.5]{EG03}.
\begin{lem} \label{l:div}
  Let $\b$ be a group-like element of $H^*$.  Suppose  that $V \in \C{H}$ is
  indecomposable and $\k_\b \o  V \cong V$. Then $\ord(\b)
  \mid \dim V$.
\end{lem}
\begin{proof}
  Note that $\k_\b \o V \cong \ld{ R(\b)} V$ as
  $H$-modules.
  Let $n=\ord(\b)$ and
  $\eta: \ld{ R(\b)} V \to  V$ an isomorphism of $H$-modules.
  Since $R(\b)^n=\id_H$, $R(\b)$ is diagonalizable.
  Suppose $h \in H$ is an eigenvector of $R(\b)$. Then $R(\b)(h)=\w h$ for some $n$-th root of
  unity $\w \in \k$, and
  $$
  \w\eta(h v)= \eta(R(\b)(h) v) = h\eta(v)
  $$
  for all $v \in V$. Thus, $\eta^n$ is an $H$-module automorphism on
  $V$. Since $V$ is finite-dimensional and indecomposable,
  $\End_H(V)$ is a finite-dimensional local algebra over $\k$ (cf. \cite{PierceBook}).
  Since $\k$ is algebraically closed, $\eta^n= c \cdot \id_V$ for some $c \in
  \k$. By dividing $\eta$ with an $n$-th root of $c$,
  one may assume $\eta^n=\id_V$. Then $V$ is a left $\k[\b]$-module with the action given by
  $$
  \b v  = \eta(v), \quad v \in V.
  $$
  Define the right $H^*$-comodule structure $\rho: V \to V \o H^*$, $\rho(v)=\sum
  v_0 \o v_1$ by the equation
  $$
  hv =\sum v_0 v_1(h) \quad \text{for all } h \in H\,.
  $$
  It is straightforward to check that  $\rho(\b v) = \b \rho(v)$ for $v \in V$, and so $V \in \ld
  {\k[\b]}\MM^{H^*}$. By the Nichols-Zoeller Theorem, $V$ is a free
  $\k[\b]$-module. In particular, $\ord(\b)$ divides $\dim V$.
\end{proof}

\section{Non-semisimple Hopf algebras of dimension $pq$}
\emph{Throughout the remaining discussion, we will assume $H$ to be
a non-semisimple Hopf algebra over $\k$ of dimension $pq$, where $p,
q$ are primes and $2 < p < q$. The antipode of $H$ will continue to
be denoted by $S$.} By \cite{LaRa88}, $H^*$ is also a non-semisimple
Hopf algebra of dimension $pq$ with antipode $S^*$.

In this section, we will obtain a lower bound for the dimensions of
certain indecomposable $H$-modules $V$ which satisfy $V \cong \k_\b
\o V\bidu$ for some $\b \in G(H^*)$ (\nmref{c:dim}). We begin with
the following lemma:
\begin{lem}\label{l:basic} $H, p, q, S$ as above. Then:
  \enumeri{
    \item $H$ or $H^*$ is not unimodular.
    \item $|G(H)|=1$ or $p$.
    \item $\ord S^2 = 2p$.
    \item Let $C$ be a subcoalgebra of $H$ invariant under ${S}^{2p}$ and $C \nsubseteq \k
    G(H)$. Then ${S}^{2p}|_C \ne \id_C$.
    \item Let $V$ be an indecomposable $H$-module with $\dim V >
    1$. Then there exist $h \in H, v \in V$ such that $hv \ne
    S^{2p}(h)v$.
    \item There exists a simple $H$-module $V$ such that
    $$
    \dim V >1 \text{ and } \dim P(V) \ge 2\dim V.
    $$
  }
\end{lem}
\begin{proof} (i) Since $\dim H$ is odd, it follows from \cite[Theorem 2.2]{LaRad95} that $H$ or $H^*$ is
not unimodular.  \smallskip \\
(ii) follow immediately from \cite[Proposition 5.2]{Ng02}. \smallskip \\
(iii) By \cite[Proposition 5.2 and Theorem 5.4]{Ng02}, $\ord S^4 =p$
and $\Tr(S^{2p})=p^2 d$ for some odd integer $d$. If $S^{2p}=\id$,
then $pd=q$ which contradicts that $p,q$ are distinct primes.
Therefore, $\ord S^2 = 2p$ and  this proves (iii).
\smallskip \\
 (iv) Let $B$ be the  subalgebra of $H$ generated by $C$.
Then $B$ is a bialgebra and hence a Hopf subalgebra of $H$. In
particular, $\dim B =1, p, q$ or $pq$. If $\dim B\in \{1, p, q\}$, then $B$
is a group algebra (cf. \cite{Zhu}) and hence $C \subseteq B
\subseteq \k G(H)$. The assumption of $C$ forces  $\dim B=pq$ or
equivalently $B=H$. If ${S}^{2p}|_C = \id_C$, then ${S}^{2p}|_B
=\id_B$ and this contradicts (iii). Therefore, ${S}^{2p}|_C \ne
\id_C$.
\smallskip \\
(v) Suppose that $hv = S^{2p}(h) v$ for all $h \in H$ and $v \in V$.
Let $I \subseteq H$ be the annihilator of $V$, and $\pi : H \to H/I$
the natural surjection. Then $S^{2p}(h)-h \in I$ for $h \in H$, and
so $\pi \circ S^{2p} = \pi$. Hence we have the commutative diagram
of coalgebra maps:
$$
 \xymatrix{
  & H^* \ar[dd]^-{(S^{2p})^*=(S^*)^{2p}\,.}\\
 (H/I)^* \ar[rd]_-{\pi^*} \ar[ru]^-{\pi^*} & \\
    & H^*
 }
 $$
Let $C =\pi^*(H/I)^*$. Then the commutative diagram implies that
$(S^*)^{2p}(C) = C$ and $(S^*)^{2p}|_C =\id_C$. Since $V$ is an
indecomposable $H/I$-module and $\dim V
>1$, $H/I$ is not a commutative semisimple algebra. Hence $C$ is not cosemisimple
cocommutative. In particular, $C \nsubseteq \k G(H^*)$. However,
this contradicts (iv).\smallskip\\
(vi) Suppose the statement is false. Then, by Lemma \ref{p:nakayama
perm}, every simple $H$-module $V$ with $\dim V >1$ is projective.
Thus the composition factors of $P(\k)$ are all 1-dimensional. By
\cite[Lemma 2.3]{EG03}, $P(\k)$ is a direct sum of 1-dimensional
$H$-modules. Hence $\k=P(\k)$. However, this contradicts \cite[Lemma
2.2]{EG03} that $\k$ is not projective.
\end{proof}

\begin{lem}\label{l:basicpq}
  Let $\sigma = S^2 \circ R(\b)$ for some $\b \in
  G(H^*)$, and let  $V \in \C{H}$ be indecomposable with $\dim V > 1$ and $V
  \cong \ld{\sigma} V$. Then:
  \enumeri{
  \item For any isomorphism  $\phi :
  V \to \ld{\sigma} V$ of $H$-modules, $\phi^p \ne \id$.
   \item There exists an $H$-module
  isomorphism $\psi: V \to \ld{\sigma} V$ such that $\psi^{2p}=\id$
  and $\Tr(\psi^p)$ is a non-negative integer.}
\end{lem}
\begin{proof}
   (i) By Lemma \ref{l:basic}(ii), $\s^p = S^{2p}$.
   Suppose  there exists an $H$-module map $\phi: V \to \ld{\sigma}V$ such
  that $\phi^p=\id$.
  Then  for all $h \in H$ and $v \in V$,
  $$
  hv=\phi^p(hv) = \sigma^p(h)\phi^p(v)=  \sigma^p(h)v = S^{2p}(h) v\,.
  $$
  This contradicts Lemma \ref{l:basic}(v). \smallskip\\
  (ii) Let $\phi: V \to \ld\s V$ be an isomorphism of $H$-modules.
 Then  for all $h \in H$ and $v \in V$, we have
  $$
  hv=\phi^{2p}(hv) = \sigma^{2p}(h)\phi^{2p}(v)\,.
  $$
  By Lemma \ref{l:basic}(iii), $\s^{2p}=S^{4p}=\id$, so $\phi^{2p}$ is
  an $H$-module automorphism on $V$. Since $V$ is indecomposable, $\End_H(V)$ is a
  finite-dimensional local $\k$-algebra (cf. \cite{PierceBook}). In particular, $\phi^{2p}=
  c \id_V$ for some non-zero $c \in \k$. Let $t \in \k$
  be a $2p$-th root of $c$ and $\ol \phi = t\inv\phi$. Then $\ol \phi$ is
  also an isomorphism of $H$-module from $V$ to $\ld{\s} V$ and
  $(\ol \phi)^{2p}=\id$. In particular, $\Tr(\ol \phi^p)$ is an
  integer. Set $\psi = - \ol \phi$ if $\Tr(\ol \phi^p)< 0$, and $\psi=\ol \phi$
   otherwise. Then $\psi$ is a required isomorphism.
\end{proof}

We close this section with following corollary which is an enhanced
result of \cite[Lemma 2.11]{EG03}.

\begin{cor}\nmlabel{Corollary}{c:dim} Let $V$  an indecomposable $H$-module of odd dimension which
satisfies one of the following conditions:
\begin{enumerate}
\item[\rm(I)] $V \cong \k_{\b}\o V\bidu$ for some non-trivial element $\b \in
G(H^*)$,
\item[\rm(II)] $V\cong V\bidu$ and $V$ is projective.
\end{enumerate}
Then $\dim(V) \ge p+2$.
  \end{cor}
\begin{proof}
  Notice that any 1-dimensional $H$-module does not satisfy (I) or
  (II). Therefore, if $V$ satisfies condition (I) or (II), then $\dim V >1$.
  Since $V \cong  \k_\b \o V\bidu$, $V \cong \ld{\sigma}V$ where $\s =S^2 \circ R(\b)$. By
  Lemma \ref{l:basicpq}, there exists an $H$-module isomorphism $\psi: V \to \ld \s V$ such that
  $\psi^{2p}=\id_V$ and $\Tr(\psi^p) \ge 0$. Since $V$ satisfies (I) or (II), it follows from
  Lemmas \ref{l:trace0} and \ref{l:2ndtrace0} that $\Tr(\psi)=0$. By a linear algebra
  argument, $\Tr(\psi^p)=pd$ for some non-negative odd integer $d$ (cf.
  \cite[Lemma 1.3]{Ng02}). This forces
  $\dim V \ge p$. Suppose
  $\dim V=p$. Then $\Tr(\psi^p)=p$ and hence $\psi^p=\id_V$.
  However, this contradicts Lemma \ref{l:basicpq} (i). Therefore,
  $\dim V>p$.
\end{proof}

\section{Orbits of simple modules}\label{s:2-orbits}
\emph{We continue to assume that $H$ is a non-semisimple Hopf
algebra of dimension $pq$ with $p, q$ prime and $2 < p< q$. By
duality and Lemma \ref{l:basic}(i), we can further assume that the
distinguished group-like element $\a \in H^*$ is non-trivial}.
Hence, by Lemma \ref{l:basic}(ii),   $\langle \a \rangle=G(H^*)$ and
$\ord \a = |G(H^*)|=p$.

Let us denote $[V]$ for the isomorphism class of an $H$-module $V$,
and $\Irr(H)$ the set of all isomorphism classes of simple $H$-modules.
One can define a left action of $G(H^*)$ on $\Irr(H)$ as follows:
$\b[V]= [\k_\b\o V]$ for $[V]\in \Irr(H)$. A right action of
$G(H^*)$ on $\Irr(H)$ can be defined similarly. The orbits of $[V]
\in \Irr(H)$ under these $G(H^*)$-actions are respectively denoted
by $O_l(V)$ and $O_r(V)$. A simple $H$-module $V$ is called
\emph{regular} if $O_l(V)=O_r(V)$. It is easy to see that $\k_\b$ is
left unstable and regular for $\b \in G(H^*)$.

 Following the terminology of \cite{EG03}, a simple
$H$-module $V$ is called \emph{left stable} if $\k_\a \o V \cong V$,
i.e. $O_l(V)$ is a singleton, or otherwise \emph{left unstable}. We
can define \emph{right stable} (resp. \emph{right unstable})
$H$-modules similarly. The set of all left (resp. right)
$G(H^*)$-orbits in $\Irr(H)$ is denoted by $\OO_l(H)$ (resp.
$\OO_r(H)$). In this section, we obtain in Corollary \ref{c:ge3}
that $|\OO_l(H)| \ge 3$. We also show in Proposition \ref{p:ge4}
that if every simple $H$-module is regular and left unstable, then
$|\OO_l(H)|\ge 4$.

\begin{remark} \label{r:3.1}
  Since
  $$
  O_r(V\du)=\{[W\du] \mid [W] \in O_l(V)\}
  $$
  for $[V] \in \Irr(H)$, we find $|\OO_l(H)|=|\OO_r(H)|$. Moreover, $P(\k_\b \o V) \cong \k_\b \o P(V)$
and $P(V\o \k_\b) \cong P(V)\o
  \k_\b$. Therefore,
  $\dim P(W) = \dim P(V)$ for all $[W]\in O_l(V) \cup O_r(V)$. In particular, $\dim P(\k)=\dim P(\k_\b)$
  for all $\b \in G(H^*)$.
\end{remark}

There is a lower bound for the dimensions of the left stable
indecomposable projective $H$-modules.
\begin{lem} \label{l:lb}
  If $V$ is a left stable simple $H$-module, then $\dim P(V) \ge 2p$.
\end{lem}
\begin{proof}
 By Lemma \ref{l:div},
 $\dim V =np$ for some positive integer $n$. If $\dim P(V) \ge 2\dim V$ then $\dim P(V) \ge 2p$.
 Now we assume $\dim P(V) < 2\dim V$. By Lemma \ref{p:nakayama perm} that $P(V)= V$ and hence
 $V \cong \k_\a \o V\bidu$.
 By \nmref{c:dim}, $\dim V \ne p$ and hence $\dim P(V)=\dim V \ge 2p$.
\end{proof}

Recall that if $Q$ is a projective $H$-module, then $Q$ is a direct
sum of indecomposable projective $H$-modules. More precisely,
\begin{equation} \label{eq:mult}
 Q \cong \bigoplus_{[V] \in \Irr(H)} N^Q_V \cdot P(V)
\end{equation}
where the multiplicity $N^Q_V=\dim \Hom_H(Q, V)$. Note that if $V,W$
are simple, then
\begin{equation} \label{eq:delta}
\dim \Hom_H(P(V), W)=\delta_{[V], [W]}.
\end{equation}
For all $X,Y
\in \C{H}$, $X \o Q$ and $Q \o Y$ are projective and we have the
natural isomorphisms of $\k$-linear spaces
\begin{equation} \label{eq:hom}
  \Hom_H(X\du \o Q, Y) \cong \Hom_H(Q, X \o Y) \cong \Hom_H(Q \o \ldu Y, X)\,.
\end{equation}

\begin{lem} \label{l:mult}
  For $[V] \in \Irr(H)$ and $\b \in G(H^*)$, we have
 $$
  \dim \Hom_H(P(V) \o \ldu V, \k_\b) \cong  \left\{\begin{array}{ll}
   \displaystyle \delta_{\e, \b} &
   \text{if $V$ is left unstable},\\
   \displaystyle 1 & \text{if  $V$ is left stable}.
  \end{array}\right.
  $$
  \end{lem}
\begin{proof}
  By \eqref{eq:hom} and \eqref{eq:delta},
  we have
  \begin{eqnarray*}
     \dim \Hom_H(P(V) \o \ldu V, \k_\b)&=&\dim \Hom_H(P(V),  \k_\b \o V) \\
     &=&
  \left\{\begin{array}{ll}
   \delta_{\e, \b} & \text{if } V \text{ is left unstable},\\
   1 & \text{if } V \text{ is left stable}.
  \end{array}\right.\qedhere
  \end{eqnarray*}
  \end{proof}

The following corollary gives a lower bound for the dimension
contributed by a left or right $G(H^*)$-orbit.
\begin{cor} \label{c:orbitcount}
For all simple $H$-module $V$, we have
  $$
  \sum_{[W]\in O(V)} \dim W \dim P(W) \ge p\cdot \dim P(\k)
  $$
  where $O(V)=O_l(V)$ or $O_r(V)$.
\end{cor}
\begin{proof}
  If $V$ is left unstable, then Lemma \ref{l:mult} implies that $P(\k)$ is a summand of $P(V)\o \ldu V$ and so
  $$
  \sum_{[W]\in O_l(V)} \dim W \dim P(W) =p\cdot \dim (P(V)\o \ldu V) \ge p\cdot \dim P(\k)\,.
  $$
  If $V$ is left stable, then  $\bigoplus\limits_{\b \in G(H^*)}P(\k_\b)$ is a summand of $P(V)\o \ldu V$. Hence, by Remark \ref{r:3.1},
  $$
  \sum_{[W]\in O_l(V)} \dim W \dim P(W) = \dim (P(V)\o \ldu V) \ge p\cdot \dim P(\k)\,.
  $$
  Since $O_r(V) = \{[W\du]\mid [W] \in O_l(\ldu V)\}$, we have
  \begin{multline*}
    \sum_{[W]\in O_r(V)} \dim W \dim P(W) =
  \sum_{[W]\in O_l(\ldu V)} \dim W\du \cdot \dim P(W\du) \\ = \sum_{[W]\in O_l(\ldu V)} \dim W  \cdot \dim P(W)
  \ge p\cdot \dim P(\k)\,. \qedhere
  \end{multline*}
\end{proof}

 Let $\{V_0, \cdots, V_\ell\}$ be a set of simple $H$-modules such that
$O_l(V_0), \dots, O_l(V_\ell)$ are all the disjoint orbits in
$\Irr(H)$ with $V_0 =\k$. By Remark \ref{r:3.1}, $\dim P(V_i) =\dim
P(W)$ for $[W]\in O_l(V_i)$. Let us simply denote $d_i$ and $D_i$
respectively for $\dim V_i$ and $\dim P(V_i)$. Obviously, $D_i\ge
d_i \ge 2$ for $i \ne 0$. Since $H$ is a Frobenius algebra, we  have
\begin{equation} \label{eq:dim H}
\begin{aligned}
\dim H   &= \sum_{[V]\in \Irr(H)} \dim V \cdot \dim P(V) \\
&= p\sum_{\text{unstable } V_i} d_i D_i + \sum_{\text{stable }V_i}
d_i D_i
\end{aligned}
\end{equation}
(cf. \cite[61.13]{CR62book}). Now we can show that $\ell \ge 3$.
  \begin{cor}\label{c:ge3}
   $|\OO_l(H)|, |\OO_r(H)| \ge 3$.
  \end{cor}
  \begin{proof} By Remark \ref{r:3.1}, it suffices to show $|\OO_l(H)|\ge
  3$. By Lemma \ref{l:basic}(vi), there exists a simple $H$-module $V$ with $\dim V
  >1$. Therefore, $|\OO_l(H)|\ge 2$. Suppose there are exactly two orbits $O_l(V_0)$ and
  $O_l(V_1)$. By Lemma \ref{l:basic}(vi), $D_1 \ge 2d_1 \ge 4$.
   If $V_1$ is stable, then  \eqref{eq:dim H} becomes
   $$
   pq = pD_0+ d_1 D_1\,.
   $$
   Lemma \ref{l:mult} and \eqref{eq:mult} imply
   $$
   d_1 D_1 = \dim (P(V_1) \o \ldu {V_1})= p D_0+ n_1 D_1
   $$
   for some non-negative integer $n_1 < d_1$. By eliminating $D_0$, we find
   $$
   q = (2d_1-n_1)\frac{D_1}{p}\,.
   $$
   Since $P(V_1)$ is also stable, $D_1$ is
   divisible by $p$. This forces $2d_1-n_1 =q$ and $D_1 =p$ which contradicts Lemma \ref{l:lb}. Therefore,
   $V_1$ must be unstable. Now,  \eqref{eq:dim H} becomes
   $$
   pq = pD_0+ pd_1 D_1\,.
   $$
   Lemma \ref{l:mult} and \eqref{eq:mult} imply
   $$
   d_1 D_1 = \dim (P(V_1) \o \ldu {V_1})=  D_0+ n_1 D_1
   $$
   for some non-negative integer $n_1 < d_1$. By eliminating $D_0$, we find
   $$
   q = (2d_1-n_1)D_1\,.
   $$
   Since  $2d_1-n_1 >1$ and  $D_1 \ge 4$, the above equality contradicts that $q$ is a prime. Therefore, $|\OO_l(H)| \ne 2$ and hence  $|\OO_l(H)| \ge 3$.
  \end{proof}

We close this section with the following proposition.

\begin{prop}\label{p:ge4}
  If every simple $H$-module is left unstable and regular, then $|\OO_l(H)|\ge 4$.
\end{prop}
\begin{proof}
  By Corollary \ref{c:ge3}, it suffices to show $|\OO_l(H)|\ne 3$. Suppose that $|\OO_l(H)|= 3$.
  By  \eqref{eq:mult}, \eqref{eq:dim H} and Lemma \ref{l:mult},
  we have the equations:
  \begin{eqnarray}
    \label{eq1} pq&=& p D_0+ pd_1 D_1 +p d_2 D_2,\\
    \label{eq2} d_1 D_1 &=& D_0 + N_{11}^1 D_1 +  N_{11}^2 D_2,\\
    \label{eq3} d_2 D_2 &=& D_0 + N_{22}^1 D_1 +  N_{22}^2 D_2,
  \end{eqnarray}
  where $N_{ii}^j = \sum\limits_{\b \in G(H^*)}\dim \left(\Hom_H(P(V_i)\o \ldu V_i, \k_\b \o V_j)\right)$.
  On the other hand,
  if $i \ne j$, then $[V_i] \not\in O_r(V_j)$ since $V_j$ is regular. Therefore,
  $$
  \Hom_H(V_j\du \o P(V_i), \k_\b)\cong \Hom_H(P(V_i), V_j \o \k_\b) = 0
  $$
  for all $\b \in G(H^*)$, and so we have
  \begin{eqnarray}
   \label{eq4}  d_2 D_1 &=& M_{21}^1 D_1 +  M_{21}^2 D_2,\\
    \label{eq5} d_1 D_2 &=& M_{12}^1 D_1 +  M_{12}^2 D_2,
  \end{eqnarray}
  where $M_{ji}^k = \sum\limits_{\b \in G(H^*)}\dim \left(\Hom_H(V_j\du \o P(V_i),  \k_\b \o V_k)\right)$.
  By \eqref{eq:hom}, we find
  \begin{eqnarray*}
    M_{ji}^k & = & \sum\limits_{\b \in G(H^*)}\dim \left(\Hom_H(P(V_i)\o \ldu V_k, V_j \o \k_\b)\right) \\
    &=& \sum\limits_{\b \in G(H^*)}\dim \left(\Hom_H(P(V_i)\o \ldu V_k, \k_\b\o V_j )\right)=N_{ik}^j\,.
  \end{eqnarray*}
  The second equality is a consequence of the regularity of $V_j$. Now,  \eqref{eq4} and \eqref{eq5} become
  \begin{eqnarray}
   \label{eq6}  d_2 D_1 &=& N_{11}^2 D_1 +  M_{21}^2 D_2,\\
    \label{eq7} d_1 D_2 &=& M_{12}^1 D_1 +  N_{22}^1 D_2.
  \end{eqnarray}
  Equations \eqref{eq1}, \eqref{eq2} and \eqref{eq3} imply
  \begin{eqnarray}
   \label{eq8} q&=& (2d_1 -N_{11}^1) D_1 + (d_2 - N_{11}^2) D_2, \\
   \label{eq9} q&=& (d_1 -N_{22}^1) D_1 + (2d_2 - N_{22}^2) D_2,\\
   \label{eq10} 1 & < & 2d_i - N_{ii}^i\quad \text{and}\quad q > D_i \quad \text{for }i=1,2.
  \end{eqnarray}
  In particular, $D_1, D_2$ are relatively prime. It follows from \eqref{eq6} and \eqref{eq7} that
$$
D_1 \mid d_1 - N_{22}^1, \quad D_2 \mid d_2 - N_{11}^2, \quad d_1 \ge N_{22}^1,\quad d_2 \ge N_{11}^2.
$$
By Lemma \ref{l:basic} (vi), $D_i> d_i$ for some $i=1,2$. If $D_1
> d_1$, then $d_1=N_{22}^1$. Hence, by \eqref{eq9}, $q=(2d_2 -
N_{22}^2) D_2$. Similarly, if $D_2 > d_2$, then $q=(2d_1 - N_{11}^1)
D_1$. However, both of these conclusions contradict that $q$ is a
prime number.
\end{proof}

\section{The case $q \le 4p+11$}
In this section, we prove our main result:
\begin{thm}\label{t:main}
  Every Hopf algebra of dimension $pq$ over $\k$, where $p,q$ are odd primes with $p < q \le 4p+11$, is trivial.
\end{thm}

By \cite{EG99}, it suffices to show that Hopf algebras of these
dimensions are semisimple. We proceed to prove that by
contradiction. Suppose there exists a non-semisimple Hopf algebra
$H$ of these dimensions. By duality and Lemma \ref{l:basic}(i), we
can further assume the distinguished group-like element $\a \in H^*$
is not trivial.

Let $D_0$ denote $\dim P(\k)$. Since the composition factors of
$P(\k)$ cannot be all 1-dimensional (cf. \cite[Lemma 2.3]{EG03}),
$D_0 \ge 4$. It follows from Lemma \ref{c:orbitcount} that
\begin{equation}\label{eq:olb}
  \sum_{[W] \in O(V)} \dim W \dim P(W) \ge pD_0 \ge 4p
\end{equation}
for all simple $H$-modules $V$.
\begin{lem} \label{l:stablebound}
  A simple $H$-module $V$ is left stable if, and only if, $V$ is right stable. In this case, $\dim V \ge 2p$.
\end{lem}
\begin{proof}
  Let $V$ be a left stable simple $H$-module.
  By Lemma \ref{l:div}, $\dim V=np$ for some positive integer $n$. It follows from Lemma \ref{l:lb} that $\dim V\cdot \dim P(V) \ge 2p^2$.

  Notice that $W$ is left stable for $[W] \in O_r(V)$.
  By Corollary \ref{c:ge3}, there exists a right $G(H^*)$-orbit different from $O_r(V), O_r(\k)$.
  If $V$ is not right stable, then by Corollary \ref{c:orbitcount} and \eqref{eq:olb},
 \begin{eqnarray*}
   \dim H &\ge& \sum_{[W]\in O_r(V)} \dim W \dim P(W) + pD_0 +pD_0 \\
   &\ge & p(2p^2)+8p > 4p^2+11p\,.
 \end{eqnarray*}
 Therefore, $V$ is also right stable. Conversely, if $V$ is right stable, then $\ldu V$ is left stable. Hence,
 by the first part of the proof, $\ldu V$ is also right stable. Therefore, $V \cong (\ldu V)\du$ is left
 stable.

 Now let $V$ be a left stable simple $H$-module. Since $V$ is also right stable, $V\quaddu \cong
 \k_{\a\inv} \o V \o\k_{\a} \cong V$. Consider the set $\AA=\{[V], [V\du], [V\bidu], [V^{\vee\vee\vee}]\}$.
 Since $[V\quaddu]=[V]$ and $S^{4p}=\id$, the cyclic group $C_{4p}=\langle x \rangle$ of order $4p$ acts
 transitively on $\AA$ with
 $$
 x \cdot [V] = [V\du]\,.
 $$
Thus $|\AA|=1, 2$ or $4$. If $|\AA|=4$, then
$$
\dim H \ge 4  \dim V  \dim P(V) + p D_0 \ge 8p^2+4p > 4p^2+11p\,.
$$
Therefore, $|\AA|\le 2$, and so $[V\bidu]=x^2 \cdot [V] =[V]$. Thus
we obtain $\k_\a \o V\bidu \cong V$. By \nmref{c:dim}, $\dim V \ne
p$ and hence $\dim V \ge 2p$.
\end{proof}

\begin{lem}\label{l:unstable}
  $H$ has no left or right stable simple $H$-module.
\end{lem}
\begin{proof}
   By Lemma \ref{l:stablebound}, it suffices to show that left stable simple $H$-modules do not exist. Suppose there is a left stable simple $H$-module $V$.
  \enumera
  {
  \item \emph{$V$ is projective}. For otherwise, by Proposition \ref{p:nakayama perm} and Lemma \ref{l:stablebound},
   $P(V) \ge 2\dim V \ge 4p$ and so
  $$
  \dim H \ge \dim V\dim P(V) +pD_0 \ge 8p^2+4p > 4p^2+11p\,.
  $$
  \item \emph{$V$ is the unique left stable simple $H$-module.} If there exists a left stable simple
  $H$-module $W$ not isomorphic to $V$, then $\dim V, \dim W \ge 2p$
  and so
  $$
  \dim H \ge \dim V\dim P(V) +\dim W \dim P(W)+ pD_0 \ge 8p^2+4p\,.
  $$
  \item Since $V$ is projective, by Lemma \ref{l:basic}(vi), there exists a simple $H$-module $W$ such that
  $\dim P(W)\ge 2\dim W \ge 4$. Obviously, $O_l(W)$ is different from $O_l(V)$, $O_l(\k)$. By (2), $W$ is left unstable. Therefore,
  $$
  \sum_{[W']\in O_l(W)} \dim W'\dim P(W') =p \dim W \dim P(W) \ge p(2\cdot 4) =8p
  $$
  and hence
  $$
  \dim H \ge \dim V \dim P(V) + 8p +pD_0\ge 4p^2+12p,
  $$
  a contradiction! \qedhere
  }
\end{proof}
\begin{lem} \label{l:ge4}
  $|\OO_l(H)| \ge 4$ and $D_0 \le p+2$.
\end{lem}
\begin{proof}
   To show the first inequality, by Proposition \ref{p:ge4} and Lemma \ref{l:unstable},
   it suffices to prove that every simple $H$-module is regular.
   Suppose there exists a simple $H$-module $V$ such that
   $O_r(V) \ne O_l(V)$. Then $\dim V \ge 2$ and the set
   $$\AA = \{[\k_\b \o V \o \k_{\b'}]\mid \b, \b' \in G(H^*)\}$$
   contains more than $p$ elements. Obviously, the group $G(H^*) \times G(H^*)$ acts transitively on $\AA$.
   Therefore, $|\AA| =p^2$. If $V$ is not
   projective, then $\dim V \cdot \dim P(V) \ge 2 \cdot 4$ and hence
   \begin{multline*}
      \dim H \ge \sum_{[W] \in \AA} \dim W \dim P(W) +p D_0 \\
   = p^2 \dim V \dim P(V) + pD_0
   \ge 8 p^2+4p > 4p^2+11p\,.
   \end{multline*}
   Therefore $V$ is projective, and hence $\k_\b \o V \o \k_{\b'}$ are projective for $\b, \b' \in G(H^*)$.
   By Lemma \ref{l:basic}(vi), there exists a simple
   $H$-module $W$ such that $\dim W \ge 2$ and $\dim P(W) \ge 2\dim W$.
   Thus,
   $$
   \dim H \ge p^2 \dim V \dim P(V) + p \dim W \dim P(W)+ pD_0
   \ge 4p^2 + 8p+4p,
   $$
   which is a contradiction! Therefore, every simple $H$-module is regular.

   By Corollary \ref{c:orbitcount} and the first inequality, we have
   $$
   4p^2+11p \ge \dim H \ge 4 pD_0\,.
   $$
   Therefore, $D_0 \le p+2$.
\end{proof}
\begin{lem} \label{l:conjuagtebound}
  If $V$ is a simple $H$-module such that $V\cong \k_{\a\inv}\o V\o \k_\a$, then $\dim P(V) \ge 2\dim V$.
\end{lem}
\begin{proof}
  Suppose $V$ is an $H$-module such that $\k_{\a\inv} \o V \o \k_\a \cong V$ and $\dim P(V) < 2 \dim V$.
  By Lemma \ref{p:nakayama perm}, $V$ is projective and  $V \cong \k_\a \o  V\bidu$.
  Since $V\quaddu \cong \k_{\a\inv} \o V \o \k_\a$ (cf.
  \eqref{eq:swap}), we have
  $$
  V\bidu \cong \k_\a \o  V\quaddu \cong V \o \k_\a\,,
  $$
  and hence
  $$
  V \cong V\quaddu \cong  V\bidu \o \k_\a \cong V \o \k_{\a^2}\,.
  $$
  Since $\ord(\a)=p$,  $V$ is right stable, but this
  contradicts Lemma \ref{l:unstable}.
\end{proof}
Using an argument similar to \cite[Lemma 2.8]{EG03}, we find
$V\quaddu \cong V$ for all $V \in \C{H}$ in the following lemma.
\begin{lem} \label{l:conjuagtebound2}
For all simple $H$-module $V$,
  $V\quaddu \cong V$ and $\dim P(V) \ge 2\dim V$.
\end{lem}
\begin{proof}
  By \cite[Lemma 2.3]{EG03}, $P(\k)$ has a simple constituent $U$ with $\dim U > 1$.  Since $P(\k) \cong \k_{\a\inv} \o P(\k) \o \k_{\a}$,
  $\k_{\b\inv} \o U \o \k_{\b}$ are simple constituents of $P(\k)$ for $\b \in G(H^*)$. By Lemma \ref{l:ge4}, $\dim P(\k) \le p+2 < 2p+2$. Therefore, $U$
  must be invariant under the conjugation by $\k_\a$.

  Let $E$ be the full subcategory of $\C{H}$ consisting of those $H$-modules with composition factors
  invariant under the conjugation by $\k_\a$. Let $V,W$ be simple objects in $E$.
  Suppose $V \o W \not \in E$. Then $V \o W$  has a simple constituent $X$ which is not invariant under conjugation by $\k_\a$.
  Then, $[X] \not\in O_l(\k)$ and
  $\k_{\b^{-1}} \o X \o \k_{\b}$ are composition factors of $V\o W$ for $\b \in G(H^*)$. In particular, $\dim X >1$,
  and hence
  $\dim V \o W \ge p\dim X \ge 2p$. Without loss of generality, we may assume that $\dim V \ge \sqrt{2p}$.
  By Lemma \ref{l:conjuagtebound}, $\dim P(V) \ge 2\dim V \ge 2\sqrt{2p}$. Hence, by Lemma \ref{l:ge4} and
  Corollary \ref{c:orbitcount},
  $$
    \dim H \ge p\dim V \dim P(V) + 3  p D_0
     \ge 4p^2 + 12p\,,\quad\text{a contradiction!}
  $$
  Therefore, $V \o W \in E$ and so $E$ forms a tensor subcategory of $\C{H}$.
  There exists a quotient Hopf algebra $H'$ of $H$ such that $E$ is tensor equivalent to $\C{H'}$.
  If $H' \not\cong H$, then $\dim H' =1, p$ or $q$, and hence $H'$ is an abelian group algebra (cf. \cite{Zhu}).
  Thus every simple $H$-module in $E$ is one dimensional.
  This contradicts that $U \in E$. Therefore, $H'=H$ and hence every $H$-module is invariant under the conjugation
  by $\k_\a$. By \eqref{eq:swap},
  $V\quaddu \cong V$ for all $V \in \C{H}$. The second assertion follows immediately from Lemma
  \ref{l:conjuagtebound}.
\end{proof}
\begin{remark} For all simple $H$-modules $V$ with $\dim V \ge 2$,
$$
\sum_{[W] \in O_l(V)} \dim W \dim P(W) \ge p \dim V \dim P(V) \ge
8p\,.
$$
\end{remark}
\begin{lem}\label{l:D0even}
  $D_0$ is an even integer.
\end{lem}
\begin{proof}
  Suppose $D_0$ is odd. Since $P(\k)\cong P(\k)\bidu$, by \nmref{c:dim},
  $D_0 \ge p+2$.  Recall from Corollary \ref{c:orbitcount}
  that $\dim V \dim P(V) \ge D_0$ for all simple $H$-modules $V$. There exists a simple $V$ such that
  $\dim V \dim P(V) > D_0$ for otherwise we have
  $$
  pq=\dim H=p\cdot|\OO_l(H)|\cdot D_0
  $$
  which contradicts that $q$ is a prime. By Lemma \ref{l:mult}, there is a simple $H$-module $W$ with $\dim W > 1$
  such that $P(W) \oplus P(\k)$ is a direct summand of $P(V)\o \ldu V$. It follows from Lemma
  \ref{l:conjuagtebound2} that
  $$
  \dim V \dim P(V) \ge D_0+ \dim P(W) \ge D_0 + 4
  $$
  and so
  $$
   \dim H \ge 3 p D_0+ p \dim V \dim P(V) \ge p(4D_0+4) \ge 4p^2+12p.\qedhere
  $$
\end{proof}
\begin{lem} \label{l:last}
For any simple $H$-module $V$, $V\bidu \in O_l(V)$ if, and only if,
$V\bidu \cong V$.
\end{lem}
\begin{proof}
If $V\bidu \in O_l(V)$, then $V\bidu \cong
  \k_{\b}\o V$ for some $\b \in G(H^*)$. Since $V\quaddu \cong V$, we have
  $$
   V \cong (\k_\b \o V)\bidu \cong \k_\b \o V\bidu \cong \k_{\b^2} \o V .
  $$
  By Lemma \ref{l:unstable}, $V$ is left unstable. Therefore, $\b^2=\e$ and hence $\b=\e$. Thus, we have $V\bidu \cong V$. The converse of the statement is obvious.
\end{proof}

 Let $S$ be a complete set of representatives of the left
 $G(H^*)$-orbits in $\Irr(H)$ with $\k \in S$. Then
  we have
  \begin{multline*}
  pq =\dim H =\sum_{V\in S } p \dim V \dim P(V)  \\ =
  p\sum_{\substack{V \in S \\ V\bidu \not\cong V}} \dim V \dim P(V)  + pD_0 +
  p\sum_{\substack{V \in S\setminus\{\k\} \\ V\bidu \cong V}} \dim V \dim P(V)\,.
  \end{multline*}
  By Lemma \ref{l:last}, the first term of the last expression is an even
  integer. The second term is also even by Lemma \ref{l:D0even}.
  Therefore, the third term must be odd. There exists
  $X \in S\setminus\{\k\}$ such that $X \cong X\bidu$ and $\dim X \dim P(X)$ is
  odd. In particular,  $\dim X$ and $\dim P(X)$ are odd integers $> 1$.
  It follows from \nmref{c:dim} that $\dim P(X) \ge p+2$.
  If $\dim X > 3$, then
  $$
  \dim H \ge 3pD_0+p\dim X \dim P(X) \ge p(12+5(p+2))> 4p^2+11p\,.
  $$
  Therefore, $\dim X =3$. Notice that $X \o P(\k)$  has even dimension and it contains a summand
isomorphic to $P(X)$. Since $\dim P(X)$ is
odd and  $X \o P(\k) \cong (X \o P(\k))\bidu$, $X \o P(\k)$ must
have another odd dimensional indecomposable projective summand $Q$
such that $Q\bidu \cong Q$. By \nmref{c:dim}, $\dim Q \ge p+2$ and
hence
$$
3 D_0 = \dim (X \o P(\k))\ge \dim P(X) + \dim Q \ge 2(p+2).
$$
Therefore, $D_0 \ge (2p+4)/3$ and so
\begin{multline*}
  \dim H \ge p (3D_0+ \dim X  \dim P(X)) \ge p(2p+4 +3(p+2))> 4p^2 +11p\,.
\end{multline*}
This is again a contradiction! That means non-semisimple Hopf algebras of dimension $pq$, where $p$ and $q$ are distinct odd primes with $p<q\le 4p+11$,  do not exist.  This completes the proof of Theorem \ref{t:main}.
\providecommand{\bysame}{\leavevmode\hbox
to3em{\hrulefill}\thinspace}
\providecommand{\MR}{\relax\ifhmode\unskip\space\fi MR }
\providecommand{\MRhref}[2]{%
  \href{http://www.ams.org/mathscinet-getitem?mr=#1}{#2}
} \providecommand{\href}[2]{#2}

\end{document}